\documentclass[11pt]{amsart}
\usepackage{amsmath,amssymb,amsthm,enumitem}
\usepackage{graphicx}
\usepackage{hyperref}
\usepackage{amsrefs}
\newtheorem{theorem}{Theorem}
\newtheorem{proposition}[theorem]{Proposition}
\newtheorem{corollary}[theorem]{Corollary}
\newtheorem{lemma}[theorem]{Lemma}
\theoremstyle{remark}
\newtheorem{remark}[theorem]{Remark}
\newcommand{\RR}{{\mathbb R}}
\newcommand{\ZZ}{{\mathbb Z}}

\begin{document}
\title{Embedded constant mean curvature tori in the three-sphere}
\author{Ben Andrews}
\address{Mathematical Sciences Institute, Australia National University; Mathematical Sciences
Center, Tsinghua University; and Morningside Center for Mathematics, Chinese Academy of Sciences.}
\email{Ben.Andrews@anu.edu.au}
\author{Haizhong Li}
\address{Department of mathematical sciences, and Mathematical Sciences
Center, Tsinghua University, 100084, Beijing, P. R. China}
\email{hli@math.tsinghua.edu.cn}
\thanks{The research of the first author was partially supported by Discovery Projects grant DP120102462 of the Australian Research Council.
The Research of the second author was supported by NSFC No. 10971110.}
\maketitle

\begin{abstract}
We prove that any constant mean curvature embedded torus in the three dimensional sphere is axially symmetric, and use this to give a complete classification of such surfaces for any given value of the mean curvature.
\end{abstract}

\section{Introduction}

The study of constant mean curvature (CMC) surfaces  in spaces of constant curvature (that is, $\mathbb{R}^3$, the sphere $S^3$, and hyperbolic space $H^3$),  is  one of the oldest subjects in differential geometry. There are many beautiful results on this topic (see for example \cite{Ho1}, \cite {Ho2},   \cite{Ch}, \cite{We}, \cite{Ab}, \cite{PS}, \cite{Bo} and many others).

 The simplest examples of CMC surfaces in $S^3$ are the totally umbilic 2-spheres.  Another basic example is the so-called Clifford torus $T_r\equiv S^1(r)\times S^1(\sqrt{1-r^2}), 0<r<1$.
 Identifying $S^3$ with the unit sphere in $\mathbb{R}^4$,  the Clifford torus $T_r$ is defined for $0<r<1$ by
$$
T_r\equiv \Big \{ (x_1,x_2,x_3,x_4) \in S^3: x_1^2+x_2^2=r^2, x_3^2+x_4^2=1-r^2 \Big \}.
$$

By constructing a holomorphic quadratic differential for CMC surfaces, H. Hopf showed that any CMC two-sphere in $R^3$ is totally umbilical (see \cite{Ho1}). S.~S.~Chern extended Hopf's result to CMC two-spheres  in $3$-dimensional space forms (see \cite{Ch}). H.~C.~Wente was the first (see \cite{We}) to show the existence of compact immersed CMC tori in $\RR^3$. Wente's examples solved the long standing problem of Hopf (see \cite{Ho2}):  Is a compact CMC surface in $\RR^3$ necessarily a round sphere? A. D. Alexandrov  (see \cite{Al}) showed that if a compact CMC surface is embedded in $R^3$, $H^3$ or a hemisphere $S^3_+$, then it must be totally umbilical.  Wente's paper was followed by a series of papers (see \cite{Ab},\cite{PS}, \cite{Bo} and many others), where Wente tori were investigated in detail and other examples of CMC tori were constructed. In particular A. I. Bobenko (\cite{Bo}) constructed CMC tori in $\RR^3$, $S^3$ and $H^3$.

In principle, the construction in \cites{PS, Bo} gives rise to all CMC tori in $S^3$, but reading off properties such as embeddedness from this classification is difficult.
It remains an interesting unsolved problem to classify all embedded CMC tori in $S^3$.  It is explicitly conjectured by Pinkall and Sterling \cite{PS}*{p.250} that such surfaces are surfaces of revolution, and our main result confirms this.  We then give a complete classification of such embedded tori by completing a classification of rotationally symmetric CMC surfaces by Perdomo \cite{Per}.  Our result is as follows:

\begin{theorem}
\label{main.theorem}
\begin{enumerate}[label={(\roman*)}]
\item\label{item:symmetry}  Every embedded CMC torus $\Sigma$ in $S^3$ is a surface of rotation:  There exists a two-dimensional subspace $\Pi$ of $\RR^4$ such that $\Sigma$ is invariant under the group $S^1$ of rotations fixing $\Pi$.
\item\label{item:discretesymm} If $\Sigma$ is an embedded CMC torus which is not congruent to a Clifford torus, then there exists a maximal integer $m\geq 2$ such that $\Sigma$ has $m$-fold symmetry:  Precisely, $\Sigma$ is invariant under the group $\ZZ_m$ generated by the rotation which fixes the orthogonal plane $\Pi^\perp$ and rotates $\Pi$ through angle $2\pi/m$.
\item\label{item:uniqueness}  For given $m\geq 2$, there exists at most one such CMC torus (up to congruence).
\item\label{item:existence}  For given $m\geq 2$, there exists an embedded CMC torus with mean curvature $H$ and maximal symmetry $S^1\times\ZZ_m$ if $|H|$ lies strictly between $\cot\frac{\pi}{m}$ and $\frac{m^2-2}{2\sqrt{m^2-1}}$.
\item\label{item:rigidity}
If $H\in\{0,\frac{1}{\sqrt{3}},-\frac{1}{\sqrt{3}}\}$ then every embedded torus with mean curvature $H$ is congruent to the Clifford torus.
\end{enumerate}
\end{theorem}

Our contributions are items \ref{item:symmetry} and \ref{item:uniqueness}.  Given these, the remaining results follow from \cite{Per}.  The case $H=0$ of item \ref{item:rigidity} is the Lawson conjecture proved recently by Brendle \cite{Br}.  The rigidity appearing for $H=\pm\frac{1}{\sqrt{3}}$ is unexpected, however.
We note that the embeddedness assumption in Theorem \ref{main.theorem} is crucial:   Bobenko \cite{Bo} has constructed an infinite family of non-rotationally symmetric immersed CMC tori in $S^3$ (see also Brito-Leite \cite{BL}, Wei-Cheng-Li \cite{WCL}, Perdomo \cite{Per}).

Recently the first author proved in \cite{An} a non-collapsing result for mean-convex hypersurfaces in Euclidean space evolving by mean curvature flow.  The estimate was inspired by earlier work of Weimin Sheng and Xujia Wang \cite{SW} (see also \cite{Wh}), who in the course of a detailed analysis of singularity profiles in mean curvature flow proved that for any embedded compact mean-convex hypersurface $M$ moving under the mean curvature flow, there is a positive constant $\delta$ such that at every point $x$ of $M$ there is a sphere of radius $\delta/H(x)$ enclosed by $M$ which touches $M$ at $x$.

The main contribution of \cite{An} was a very direct proof of this non-collapsing result using a maximum principle argument.  In particular, the noncollapsing condition described in the last paragraph was expressed as the positivity of a certain function of pairs of points on the hypersurface, and this function was shown to admit a maximum principle argument to preserve initial positivity.  We will discuss this expression for noncollapsing in section \ref{sec:touching-spheres}.  The idea of working with functions of pairs of points was in turn inspired by earlier work of Huisken \cite{Hu} and Hamilton \cites{HamiltonCSFIsoperim, HamiltonRFIsoperim} for the curve shortening flow and for Ricci flow on surfaces.

In \cite{ALM} this technique was adapted to a family of other flows, and it was shown that the mean curvature used to determine the scale of noncollapsing can be replaced by any function satisfying a certain differential inequality related to the linearisation of the evolution equation.

Recently, Brendle \cite{Br} made a remarkable breakthrough:  He reworked the technique of \cite{An} in such a way that it can be applied to minimal surfaces, and succeeded in proving the Lawson conjecture, that the only embedded minimal torus in $S^3$ is the Clifford torus.  Brendle used the same `non-collapsing' quantity as in \cite{An} (we describe the geometric meaning of this below in section \ref{sec:touching-spheres}), but exploited the additional special structure coming from the minimal surface condition to show that the maximum principle argument still works if  the radii of the touching spheres are compared to the length of the second fundamental form rather than the mean curvature.

In this paper, we again use the non-collapsing argument originating from \cite{An}, together with the modifications introduced by Brendle.  A crucial point is to choose carefully the scale on which to compare the touching spheres:  We show in section \ref{sec:MP} that the difference of the curvatures of the touching spheres from the mean curvature $H$ of $\Sigma$ can be compared in ratio with the difference of the largest principal curvature from $H$ using a maximum principle argument.  This implies (as in Brendle's case) that every point of the surface $\Sigma$ is touched by a ball with boundary curvature equal to the maximum principal curvature at that point.

In Brendle's argument in \cite{Br} one can touch by such spheres on both sides of the surface, and deduce from this that the second fundamental form is parallel, from which the rigidity follows easily.  In our case this is no longer true, and we can only deduce that some components of the derivative of second fundamental form vanish.  However this is enough to conclude (as we do in Section \ref{sec:symmetry}) that the surface is rotationally symmetric.

Perdomo \cite{Per} has given an analysis of rotationally symmetric surfaces with constant mean curvature in $S^3$:  These are constructed from the solutions of a certain differential equation, which are parametrized by a parameter $C$ for each $H$.  Embedded examples
arise from those solutions for which an associated `period' function $K(H,C)$ equals $2\pi/m$ for some positive integer $m$.  Perdomo showed that as $C$ varies the family of CMC surfaces deforms from the Clifford torus to a chain of `kissing spheres', and found the limiting values of $K(H,C)$.  Our second contribution in this paper is to complete the classification by proving that $K(H,C)$ is monotone in $C$, and so takes each value between the limits found by Perdomo exactly once.

It is a pleasure to thank Professor S. -T.~Yau and Professor R.~Schoen for their interest in this topic.  We also express our thanks to graduate student Mr. Zhijie Huang for assistance with checking the proof of monotonicity of $K(H,C)$ (see Proposition \ref{prop:monotonic}).

\section{Touching interior balls}\label{sec:touching-spheres}

The key geometric idea in the non-collapsing argument from \cite{An} is to compare the curvature of enclosed balls touching the surface to a suitable function (mean curvature in the setting of \cite{An}) at the touching point.  We recall here some of the expressions which arise from this picture (see also the discussion in \cite{An} and \cite{ALM}).

Let $M^n=F(\Sigma^n)$ be an embedded hypersurface in $S^{n+1}\subset\RR^{n+2}$ given by an embedding $F$, and bounding a region $\Omega\subset S^{n+1}$.   We choose $\Omega$ in such a way that the unit normal $\nu$ of $\Sigma$ points out of $\Omega$.  For given $x\in\Sigma$ we will derive an inequality which is equivalent to the geometric statement that there is a ball in $\Omega$ of boundary curvature $\Phi$ which touches at $F(x)$.   A geodesic ball in $S^{n+1}$ is simply the intersection of a ball in $\RR^{n+2}$ with $S^{n+1}$.  In particular the ball in $S^{n+1}$ with boundary curvature $\Phi$ which is tangent to $F(\Sigma)$ at the point $F(x)$ is $B=B_{\Phi^{-1}}(p)$, where $p=F(x)-\Phi^{-1}\nu(x)$, and $\nu$ is the unit normal to $F(\Sigma)$ at $F(x)$ in $S^{n+1}$ which points out of $\Omega$.  The statement that this ball lies entirely in $\Omega$ is equivalent to the statement that no other points of $F(\Sigma)$ are inside $B$.  This is in turn equivalent to the statement that for any $y\in\Sigma$, $|F(y)-p|^2\geq \Phi^{-2}$, which can be written as follows:
\begin{equation}\label{eq:defZ1}
|F(y)-(F(x)-\Phi^{-1}\nu(x))|^2-\Phi^{-2}\geq 0.
\end{equation}
Expanding the inner product (and multiplying through by $\Phi/2$) this becomes
$$
Z(\Phi,x,y):= \frac{\Phi}{2}|F(y)-F(x)|^2+\langle F(y)-F(x),\nu(x)\rangle \geq 0.
$$
Since $F(x),F(y)\in S^{n+1}$ we have $|F(x)|^2=|F(y)|^2=1$ and $\langle F(x),\nu(x)\rangle = 0$, so that
\begin{equation}\label{eq:defZ}
Z(\Phi,x,y) = \Phi(1-F(x)\cdot F(y))+\langle F(y),\nu(x)\rangle.
\end{equation}
In summary we have the following:

\begin{proposition}\label{prop:balls}
If $\Phi:\ \Sigma\to\RR$ is a smooth positive function, then the function $Z(\Phi(x),x,y)$ is non-negative for every $x,y\in\Sigma$ if and only if at every point $x\in\Sigma$ there is a ball $B\subset\Omega$ with boundary curvature $\Phi(x)$ with $F(x)\in\bar B$.
\end{proposition}

Following \cite{ALM} we call the smallest $\Phi(x)$ for which this is true the \emph{interior ball curvature} of the surface at $x$, and denote it by $\bar\Phi(x)$.  Since a ball of curvature less than the largest principal curvature cannot touch at $x$, we always have $\bar\Phi(x)\geq \lambda(x)$, where $\lambda(x)$ is the largest principal curvature of the surface at $x$.  The main result of section \ref{sec:MP} is that an embedded CMC torus in $S^3$ always has interior ball curvature equal to the maximum principal curvature at every point.  The key result of \cite{Br} is that an embedded minimal torus in $S^3$ always has interior ball curvature equal to the maximum principal curvature and exterior ball curvature equal to minus the minimum principal curvature.

\section{Interior ball curvature equals maximum principal curvature}\label{sec:MP}

Let $F: \Sigma \to S^3$ be an embedded torus in $S^3\subset\RR^4$ with constant mean curvature $H$ (note that we adopt the convention that $H$ is the average of the two principal curvatures, not their sum).  By choosing the direction of the unit normal we can assume
$H\geq 0$.  Our aim in this section is to prove that the interior ball curvature equals the maximum principal curvature, by deriving a contradiction if this is not the case.

Let $\Phi$ be a function on $\Sigma$ with $\Phi>H\geq 0$.  In the following we fix $\Phi$ and denote for brevity $Z(x,y)=Z(\Phi(x),x,y)$ (see equation \eqref{eq:defZ}).  Suppose that $Z({x},{y})\geq 0$ for all $(x,y)\in \Sigma\times\Sigma$ (so that by Proposition \ref{prop:balls} there is a ball with boundary curvature $\Phi(x)$ touching at each point $F(x)$ of the surface), and consider a pair of points $\bar{x}\neq \bar{y}$ such that $Z(\bar{x},\bar{y})=0$. Then $(\bar{x},\bar{y})$ is a minimum point of $Z$ and the differential of $Z$ at the point $(\bar{x},\bar{y})$ vanishes.

We begin by elaborating the geometric picture:
By Proposition \ref{prop:balls}, both $F(\bar{x})$ and $F(\bar{y})$ lie on the boundary of the ball $B$ in $S^3$ of boundary curvature $\Phi(\bar{x})$, which is the intersection with $S^3$ of the ball of radius $\frac{1}{\Phi(\bar{x})}$ in $\RR^4$ centred at the point $p=F(\bar{x})-\frac{1}{\Phi(\bar{x})}\nu(\bar{x})$.   Thus $B=\{z\in S^3:\ z\cdot p\geq 1\}$, so by construction $F(\Sigma)$ lies in the exterior of $B$ and touches at the points $F({\bar x})$ and $F(\bar{y})$.  It follows that the tangent spaces of the surface $F(\Sigma)$ at these points agree with those of the two-dimensional sphere $\partial B$, and the outward unit normals $\nu({\bar x})$ and $\nu(\bar{y})$ agree with those of the sphere also.   Let $\vec{\ell}=\frac{F(\bar{y})-F(\bar{x})}{|F(\bar{y})-F(\bar{x})|}$.  Then the symmetry of the sphere implies that the reflection $R_{\vec{\ell}}:\ z\mapsto z-2(\vec{\ell}\cdot z)\vec{\ell}$ maps the tangent space to $F(\Sigma)$ at $\bar{x}$ to that at $\bar{y}$, and that
\begin{equation}\label{eq:nuy}
\nu(\bar{y})=R_{\vec{\ell}}(\nu(\bar{x}))=\nu(\bar{x})-2\vec{\ell}\cdot\nu(\bar{x})\vec{\ell}
=\nu(\bar{x})+\Phi(\bar{x})d\vec{\ell},
\end{equation}
where we used the equation $Z(\bar{x},\bar{y})=0$ in the last step.

Let $(x_1,x_2)$ be geodesic normal coordinates around $\bar{x}$, and let $(y_1,y_2)$ be geodesic normal coordinates around $\bar{y}$.   We specify further that the tangent vectors
$\left\{\frac{\partial F}{\partial x^i}(\bar{x})\right\}$ diagonalize the second fundamental form,
so that $h_{11}(\bar{x}) = \lambda_1$, $h_{12}(\bar{x}) = 0$, and $h_{22}(\bar{x}) = \lambda_2$, where $\lambda_1>H>\lambda_2$.  We use the following notations:
\begin{equation}\label{eq:defs}
\mu_i=\lambda_i-H,\,\, |{\AA}(\bar{x})|^2=|{A}(\bar{x})|^2-2H^2,
\end{equation}
We also assume that the coordinate tangent vectors at
$\bar{y}$ are defined by reflecting those at $\bar{x}$:
\begin{equation}\label{eq:reflect}
\frac{\partial F}{\partial y^i}(\bar{y})=R_{\vec\ell}\left(\frac{\partial F}{\partial x^i}(\bar{x})\right).
\end{equation}

The key computation is the following:

\begin{proposition}\label{prop:D2Z}
If $\Phi(\bar{x})>\kappa_i(\bar{x})$ for all $i$, then at the point $(\bar{x},\bar{y})$ we have
\begin{align*}
\sum_i\left(\frac{\partial}{\partial x^i}+\frac{\partial}{\partial y^i}\right)^2Z
= \frac{d^2}{2}\left(\Delta\Phi-\sum_i\frac{2|\nabla_i\Phi|^2}{\Phi-\kappa_i}
+\left(|A|^2-2-2H\Phi\right)\Phi+2H\right)\Big|_{\bar{x}}.
\end{align*}
\end{proposition}

\begin{proof}
We first compute the derivatives of $Z$:
\begin{align}\label{eq:Zx}
\frac{\partial Z}{\partial x^i}&=\frac{d^2}2\frac{\partial\Phi}{\partial x^2}-\Phi d\vec{\ell}\cdot\frac{\partial F}{\partial x^i}+d\vec{\ell}\cdot h_i^p(x)\frac{\partial F}{\partial x^i}\notag\\
&=\frac{d^2}{2}\left(\frac{\partial\Phi}{\partial x^i}-\frac{2}{d}\left(\Phi\delta_i^p-h_i^p(x)\right)\vec{\ell}\cdot\frac{\partial F}{\partial x^i}\right).
\end{align}
In particular this vanishes at the point $(\bar{x},\bar{y})$, so we have
\begin{equation}\label{eq:Zx-identity}
\vec{\ell}\cdot\frac{\partial F}{\partial x^i}(\bar{x}) = \frac{d}{2}\frac{\nabla_i\Phi}{\Phi-\kappa_i}\big|_{\bar{x}}.
\end{equation}
Now we begin computing $\sum_i\left(\frac{\partial}{\partial x^i}+\frac{\partial}{\partial y^i}\right)^2Z$:  Differentiating the expression \eqref{eq:defZ1} in the direction $\frac{\partial}{\partial x^i}+\frac{\partial}{\partial y^i}$ we obtain the following:
\begin{equation}\label{eq:ZxZy}
\left(\frac{\partial}{\partial x^i}+\frac{\partial}{\partial y^i}\right)Z
=\frac{d^2}{2}\nabla_i\Phi+\left(\Phi d\vec{\ell}+\nu(x)\right)\cdot\left(\frac{\partial F}{\partial y^i}-\frac{\partial F}{\partial x^i}\right)+h_i^p(x)d\vec{\ell}\cdot\frac{\partial F}{\partial x^p}.
\end{equation}
Differentiating again and taking a sum over $i$, we obtain the following:
\begin{align}\label{eq:D2Z1}
\sum_i\left(\frac{\partial}{\partial x^i}+\frac{\partial}{\partial y^i}\right)^2Z
&=\frac{d^2}{2}\Delta\Phi+2d\nabla_i\Phi\vec{\ell}\cdot\left(\frac{\partial F}{\partial y^i}-\frac{\partial F}{\partial x^i}\right)+\Phi\left|\frac{\partial F}{\partial y^i}-\frac{\partial F}{\partial x^i}\right|^2\notag\\
&\quad\null+2h_i^p(x)\left(\frac{\partial F}{\partial y^i}-\frac{\partial F}{\partial x^i}\right)
\cdot \frac{\partial F}{\partial x^p}\\
&\quad\null+2d\vec{\ell}\cdot\nabla H(x)+h_i^pd\vec{\ell}\cdot(-h_{pi}\nu(x)-g_{ip}F(x))
\notag\\
&\quad\null+2(\Phi d\vec{\ell}+\nu(x))\!\cdot\!\left(\!H(x)\nu(x)\!+\!F(x)\!-\!H(y)\nu(y)\!-\!F(y)\!\right),\notag
\end{align}
where we used the Codazzi identity to write $\nabla_ih_i^p = \nabla^ph_{ii} = 2\nabla^pH$.
The choice \eqref{eq:reflect} and the identity \eqref{eq:Zx-identity} imply that at $(\bar{x},\bar{y})$ we have
\begin{equation}\label{eq:reflect2}
\frac{\partial F}{\partial y^i}-\frac{\partial F}{\partial x^i}=-2\vec{\ell}\cdot\frac{\partial F}{\partial x^i}\vec{\ell} = -\frac{\nabla_i\Phi}{\Phi-\kappa_i}d\vec{\ell}.
\end{equation}
The constant mean curvature condition gives $\nabla H=0$, and we also use the identities $d\vec{\ell}\cdot\nu(x) = -\frac{d^2}{2}\Phi$ and $d\vec{\ell}\cdot F(x) = -\frac{d^2}{2}$ and equations \eqref{eq:nuy} and \eqref{eq:reflect} to obtain the following by evaluating \eqref{eq:D2Z1} at $(\bar{x},\bar{y})$:
\begin{align}\label{eq:D2Z2}
\sum_i\left(\frac{\partial}{\partial x^i}+\frac{\partial}{\partial y^i}\right)^2Z
&=\frac{d^2}{2}\left(\Delta\Phi-4\frac{|\nabla_i\Phi|^2}{\Phi-\kappa_i}+2\Phi\frac{|\nabla_i\Phi|^2}{(\Phi-\kappa_i)^2}-2\kappa_i\frac{|\nabla_i\Phi|^2}{(\Phi-\kappa_i)^2}\right.\notag\\
&\quad\quad\quad\null+|A|^2\Phi+2H-2H\Phi^2-2\Phi\Big).
\end{align}
The result follows directly.
\end{proof}

\begin{corollary}
\label{key.ingredient}
At the point $(\bar{x},\bar{y})$ we have
\begin{equation*}
\sum\limits_{i=1}^2\!\left(\!\frac{\partial}{\partial x^i}\!+\!\frac{\partial}{\partial y^i}\!\right)^2\!\!\!Z
\leq \frac{d^2}{2}\left(\Delta\Phi-\frac{|\nabla \Phi|^2}{\Phi-H}+\left(|A|^2-2-2H\Phi\right)\Phi+2H\right)\Big|_{(\bar{x},\bar{y})}.
\end{equation*}

\begin{proof}
We have $\Phi-\kappa_2 = \Phi-(2H-\kappa_1)\leq \Phi+\kappa_1-2H\leq 2(\Phi-H)$.
\end{proof}
\end{corollary}

Now we can prove the main result of this section:

\begin{theorem}\label{thm:MP}
Suppose that $F: \Sigma \to S^3$ is an embedded torus with constant mean curvature $H\geq 0$ in $S^3$. Then the interior ball curvature $\bar\Phi$ of $F(\Sigma)$ is equal to the maximum principal curvature $\lambda_1$ at every point.
\end{theorem}

\begin{proof}
The case $H=0$ was proved in \cite{Br}, so we assume that $H>0$.
We will apply the maximum principle using the formula in Proposition \ref{key.ingredient}, with $\Phi$ chosen as follows:  We denote by $\lambda(x)=\lambda_1(x)$ the largest principal curvature at $x$, and by $\mu=\mu_1$ the difference $\lambda-H$.  Then we choose
$$
\Phi(x) =\kappa \mu +H
$$
where $\kappa$ is a positive constant.  We require the following variant of Simons' identity:

\begin{proposition}
\label{simons.identity}
Suppose that $F: \Sigma \to S^3$ is an embedded CMC torus in $S^3$. Then the function $\mu$ is strictly positive and satisfies the partial differential equation
$$
\Delta\mu- \frac{|\nabla \mu|^2}{\mu} + 2(\mu^2 - 1-H^2)\mu= 0.
$$
\end{proposition}

\begin{proof}
Note that $|\AA|^2 = 2\mu^2$, so $\mu$ vanishes only at umbilical points.
It follows from work of Hopf that a CMC torus in $S^3$ has no umbilical points (see \cite{Ho2} or \cite{Ch}), so $\mu$ is strictly positive and smooth everywhere.
Using the Simons identity (cf. \cite{Si}, \cite{NS}, \cite{Ya}), we have  the known result (for example, see (2.7) and (2.8) in \cite{Li})
\begin{align*}
\frac{1}{2}\Delta(|{\AA}|^2)
& =\frac{|\nabla|{\AA}|^2|^2}{2|{\AA}|^2}-|{\AA}|^4+2|{\AA}|^2+2H^2|{\AA}|^2 \\
&=2|\nabla|{\AA}||^2-|{\AA}|^4+2|{\AA}|^2+2H^2|{\AA}|^2,
\end{align*}
which is equivalent to the proposition 8.
\end{proof}

Since $\Sigma$ is compact and $\mu$ is positive, $\mu$ has a positive lower bound.  Since $F(\Sigma)$ is embedded, the interior ball curvature $\bar\Phi$ is bounded above.  Therefore for sufficiently large $\kappa$ we have $\kappa\mu+H>\bar\Phi$ and
the function $Z$ is non-negative.

Along any geodesic in $\Sigma$ through $x$ we have $\frac{2\langle d{\vec{\ell}},\nu(x)\rangle}{d^2} = -h_x(\gamma',\gamma')+O(s)$, and hence $Z(\kappa\mu+H,x,\gamma(s)) = \frac{1}{2}[\kappa\mu+H-h_x(\gamma',\gamma')]s^2+O(s^3)$.

In particular, if $\kappa<1$ then we can choose $\gamma'(0)$ to be in the direction of the largest principal curvature, so that $h_x(\gamma',\gamma')=\lambda=H+\mu$.  Then $Z(\kappa\mu+H,x,\gamma(s)) = \frac{1}{2}(\kappa-1)\mu s^2+O(s^3)$, so that $Z$ takes negative values.

On the other hand if $\kappa>1$ then in every direction from $x$ we have $Z(\kappa\mu+H,x,\gamma(s))\geq \frac{1}{2}(\kappa-1)\mu s^2+O(s^3)$, and it follows that $Z$ is positive in a neighbourhood of the diagonal $\{(x,x):\ x\in\Sigma\}$ in $\Sigma\times\Sigma$.

We choose $\bar\kappa = \inf\{\kappa>0:\ Z(\kappa\mu+H,x,y)\geq 0\text{\rm\ for\ all\ }x,y\in\Sigma\}$.  The considerations above then show that $1\leq\bar\kappa<\infty$.  Note that if $\bar\kappa=1$ then the theorem is proved, since then we have $\bar\Phi\leq\mu+H=\lambda$ and hence $\bar\Phi=\lambda$ as claimed.  We complete the proof of the theorem by assuming that $\bar\kappa>1$ and deriving a contradiction:

If $\bar\kappa>1$, then since $Z$ is positive in a neighbourhood of the diagonal in $\Sigma\times\Sigma$, there must exist a point $(\bar{x},\bar{y})$ in $\Sigma\times\Sigma$ with $\bar{x}\neq\bar{y}$ such that $Z(\bar{x},\bar{y})=0$, while $Z(x,y)\geq 0$ for every $(x,y)\in\Sigma\times\Sigma$.  Also, we have $\Phi>\kappa_i$ for all $i$, so Proposition \ref{prop:D2Z} applies.  Since this is a minimum point of $Z$, the second derivatives are non-negative.  However, using Proposition \ref{simons.identity}, the identity of Corollary \ref{key.ingredient} becomes the following:
\begin{equation}\label{eq:KIwithPhi}
\sum\limits_{i=1}^2\left(\frac{\partial}{\partial x_i}+\!\frac{\partial}{\partial y^i}\right)^2Z\big|_{(\bar{x},\bar{y})}\leq
-(\kappa^2-1)d^2\mu^2H<0,
\end{equation}
and we have a contradiction since the left hand side of \eqref{eq:KIwithPhi} is non-negative while the right-hand side is strictly negative.
\end{proof}

\section{Rotational symmetry}\label{sec:symmetry}

Now we prove that embedded CMC tori in $S^3$ have rotational symmetry:

\begin{theorem}\label{thm:symmetry}
Let $F:\ \Sigma\to S^3$ be a CMC embedding for which $Z(\lambda(x),x,y)\geq 0$ for every $x,y\in\Sigma$ (equivalently, $\bar\Phi(x)=\lambda(x)$ everywhere).  Then $\Sigma$ is rotationally symmetric.
\end{theorem}

\begin{proof}
In this case, we have
$$
Z(x,y) = \lambda(x)(1 - \langle F(x),F(y) \rangle) + \langle \nu(x),F(y) \rangle \geq 0
$$
for all points $x,y$. For simplicity, we identify the surface $\Sigma$ with its image under the embedding $F$, so that $F(x) = x$.   Since $\Sigma$ is a CMC torus and therefore has no umbilical points, we have global smooth eigenvector fields $e_1$ and $e_2$ such that $h(e_1,e_1)=\lambda_1=\lambda$ and $h(e_2,e_2)=\lambda_2=2H-\lambda$, and $h(e_1,e_2)=0$.   Exactly as in \cite{Br} we can deduce that $(\nabla_{e_1}h)(e_1,e_1)=0$, and consequently also $(\nabla_{e_1}h)(e_2,e_2)=0$ everywhere on $\Sigma$ (we repeat the argument here for convenience):  For any $x\in M$, let $\gamma(t)$ be the geodesic $\gamma(t)=\exp_x(te_1(x))$ which passes through $x$ in direction $e_1(x)$, and define $f(t)=Z(\lambda(x),x,\gamma(t))$.  Then as in \cite{Br} we have
$$
f(t)=\lambda(1-x\cdot\gamma(t))+\gamma(t)\cdot\nu(x).
$$
In particular $f(0)=0$.  Differentiating $f(t)$ we have
$$
f'(t) = \gamma'(t)\cdot (\nu(x)-\lambda x),
$$
so that $f'(0)=0$ also.  A further differentiation gives
$$
f''(t) = -[h_{\gamma(t)}(\gamma',\gamma')\nu(\gamma(t))+\gamma(t)]\cdot(\nu(x)-\lambda x),
$$
which again vanishes when $t=0$.  Since $f(t)\geq 0$ for all $t$ we must have $f'''(0)=0$:
\begin{equation}\label{eq:3.6}
0=f'''(0) = -((\nabla_{\gamma'}h)(\gamma',\gamma')\nu+h(\gamma',\gamma')D_{\gamma'}\nu+\gamma')\cdot(\nu(x)-\lambda x) = -(\nabla_{e_1}h)(e_1,e_1)\big|_x.
\end{equation}
Note that $e_i\lambda_1 = (\nabla_{e_i}h)(e_1,e_1)$, so that $e_1\lambda_1=0$ and similarly $e_1\lambda_2=0$ and $e_1\mu=0$, where $\mu=\lambda-H$.

Now we choose a local orthonormal basis $\{e_1,e_2,e_3\}$ with $e_3=\nu$ along $F(\Sigma)$, and let $\{\omega_1,\omega_2\}$ be the dual coframe of $\{e_1,e_2\}$.
We recall that the Levi-Civita connection $\omega_{12}$ is defined by
$$
de_1=\omega_{12}e_2,\,\,\,\,\,
de_2=\omega_{21}e_1, \;\;\; \omega_{12}=-\omega_{21}.
$$
We have $h_{11}=\lambda_1=\lambda=H+\mu, h_{12}=0, h_{22}=\lambda_2=H-\mu$, $\mu\neq 0$.
The derivatives of the components of the second fundamental form are defined by
\begin{equation}\label{eq:dhij}
h_{ijk}\omega_k=dh_{ij}+h_{kj}\omega_{ki}+h_{ik}\omega_{kj}.
\end{equation}
The Codazzi equations give
\begin{equation}\label{eq:codazzi}
h_{ijk}=h_{ikj}, \;\; 1\leq i,j,k\leq 2.
\end{equation}
Choosing $i=1$ and $j=2$ in \eqref{eq:dhij}, we have
\begin{equation}\label{eq:3.9}
h_{122}=(\lambda_1-\lambda_2)\omega_{12}(e_2)=2\mu\omega_{12}(e_2), \;\; h_{121}=2\mu\omega_{12}(e_1).
\end{equation}
From the equation $H={\rm constant}$, and equations \eqref{eq:codazzi}, \eqref{eq:3.9} and \eqref{eq:3.6}, we have
\begin{equation}\label{eq:3.10}
\omega_{12}(e_2)=0, \;\; 2\mu\omega_{12}(e_1)=e_2(\lambda)=e_2(\mu).
\end{equation}
Thus we have
\begin{equation}\label{eq:3.11}
\nabla_{e_1}e_1=\omega_{12}(e_1)e_2=\frac{e_2(\mu)}{2\mu}e_2,\quad
\nabla_{e_2}e_1=\omega_{12}(e_2)e_2=0.
\end{equation}

\begin{equation}\label{eq:3.12}
\nabla_{e_2}e_2=\omega_{21}(e_2)e_1=0,\quad
\nabla_{e_1}e_2=\omega_{21}(e_1)e_1=-\frac{e_2(\mu)}{2\mu}e_1.
\end{equation}
It follows from  \eqref{eq:3.12} that the flow lines of $e_2$ are geodesic in $\Sigma$.

We have the following calculations using \eqref{eq:3.11} and \eqref{eq:3.12}:
\begin{align}
1+\lambda_1\lambda_2 &=R_{1212}\notag\\
&=<\nabla_{e_1}\nabla_{e_2}e_2-\nabla_{e_2}\nabla_{e_1}e_2-\nabla_{[e_1,e_2]}e_2,e_1>\label{eq:3.13}\\
&=
e_2(\frac{1}{2\mu}e_2(\mu))
-(\frac{1}{2\mu}e_2(\mu))^2.\notag
\end{align}
Writing
\begin{equation}\label{eq:3.14}
w={\mu}^{-\frac{1}{2}},
\end{equation}
we have from \eqref{eq:3.13}
\begin{equation}\label{eq:3.15}
\frac{1}{w}e_2(e_2(w))-\frac{1}{w^4}+H^2+1=0.
\end{equation}
Multiplying by $2we_2(w)$, we obtain
\begin{equation}\label{eq:3.16}
[e_2(w)]^2+w^{-2}+(1+H^2)w^2=C_1,
\end{equation}
where $C_1$ is a constant.

Let us denote by $x:M\to S^3\subset R^4$ the position vector, and by $\bar\nabla$ the Euclidean connection on $R^4$. Using the fact that $\bar{\nabla}_{v}x=v$,  $<x,\nu(x)>=0$ and $<\nu(x),\nu(x)>=1$, we get that
$$
\begin{array}{lcl}
\bar{\nabla}_{e_1}e_1&=&\frac{e_2(\mu)}{2\mu}e_2-\lambda_1\nu-x\\
\bar{\nabla}_{e_2}e_2&=&-x-\lambda_2 \nu\\
\bar{\nabla}_{e_2}\nu &=&\lambda_2 e_2\\
\bar{\nabla}_{e_2}x&=&e_2.
\end{array}
$$
Let us fix a point $x_0\in M$, and denote by $\sigma(u)$ the geodesic in $\Sigma$ such that $\sigma(0)=x_0$ and $\sigma^{\prime}(0)=e_2(x_0)$. We write $g(u)=w(\sigma(u))$. Equation \eqref{eq:3.16} implies that
\begin{equation}\label{eq:3.17}
(g^\prime)^2+g^{-2}+(1+H^2)g^2+2H=C
\end{equation}
where $C$ is a constant greater than $2(H+\sqrt{1+H^2})$ and $C=C_1+2H$.
The polynomial
\begin{equation}\label{eq:3.18}
\xi(s)=C s^2-1-(1+H^2)s^4-2Hs^{2}
\end{equation}
is positive on an interval $(t_1,t_2)$ with $0<t_1<t_2$ and $\xi(t_1)=\xi(t_2)=0$.  The roots can be explicitly calculated:
\begin{equation}\label{eq:3.19}
\begin{aligned}
t_1&=\sqrt{\frac{C-2H-\sqrt{C^2-4 H C-4}}{2(1+H^2)}},\\
t_2&=\sqrt{\frac{C-2H+\sqrt{C^2-4 H C-4}}{2(1+H^2)}}.
\end{aligned}
\end{equation}
We have that $g$ is a periodic function with period
$$
T=2 \int_{t_1}^{t_2}\frac{t}{\sqrt{(C-2H) t^2-1-(1+H^2)t^4}}dt.
$$
We can solve $g(u)$ from (25)
$$
g(u)=\sqrt{\frac{C-2H+\sqrt{C^2-4-4HC}{\rm sin}(2\sqrt{1+H^2}u)}{2(1+H^2)}}.
$$
From the expression of $g(u)$, we get that its period $T=\frac{\pi}{\sqrt{1+H^2}}$.

\begin{lemma}\label{lem:commuting}
The vector fields $e_2$ and $\frac{e_1}{\sqrt{\mu}}$ commute.
\end{lemma}

\begin{proof}
We have using \eqref{eq:3.11} and \eqref{eq:3.12}
$$
\left[e_2,\frac{e_1}{\sqrt{\mu}}\right] = \nabla_{e_2}\left(\frac{e_1}{\sqrt{\mu}}\right)-\frac{1}{\sqrt{\mu}}\nabla_{e_1}e_2 = -\frac{e_2\mu}{2\mu^{3/2}}e_1+\frac{1}{\sqrt{\mu}}\frac{e_2\mu}{2\mu}e_1 = 0.
$$
\end{proof}

\begin{lemma}\label{lem:perp-plane}
The plane $\Pi^\perp$ generated by the vectors $e_1 $ and $\bar{\nabla}_{e_1}e_1$
is constant on $\Sigma$.
\end{lemma}

\begin{proof}
We show that the derivatives of the given basis are in $\Pi^\perp$:  Differentiating $e_1$ in the $e_1$ direction we clearly have
$$\bar{\nabla}_{e_1}e_1\in\Pi^\perp.
$$
Noting that $\bar{\nabla}_{e_1}e_1=\frac{e_2\mu}{2\mu}e_2-\lambda_1\nu-x$, we compute
\begin{align*}
\bar{\nabla}_{e_1}\left(\bar{\nabla}_{e_1}e_1\right) &=\bar{\nabla}_{e_1}\left(\frac{e_2\mu}{2\mu}e_2-\lambda_1\nu-x\right)\\
&=\frac{e_2\mu}{2\mu}\bar{\nabla}_{e_1}e_2-\lambda_1^2e_1-e_1\\
&=-\left(\frac{(e_2\mu)^2}{4\mu^2}+1+\lambda^2\right)e_1\in\Pi^\perp.
\end{align*}
Next we check the derivatives in the $e_2$ direction:  We have
$$
\bar{\nabla}_{e_2}e_1 = \nabla_{e_2}e_1-h(e_2,e_1)\nu-g(e_2,e_1)x = 0
$$
by \eqref{eq:3.11}.  Also we have
\begin{align*}
\bar{\nabla}_{e_2}(\bar{\nabla}_{e_1}e_1)
&=\bar{\nabla}_{e_2}\left(\sqrt{\mu}
\bar{\nabla}_{\frac{e_1}{\sqrt{\mu}}}e_1\right)\\
&=\frac{e_2\mu}{2\mu}\bar{\nabla}_{e_1}e_1
+\bar{\nabla}_{e_1}\bar{\nabla}_{e_2}e_1\\
&=\frac{e_2\mu}{2\mu}
\bar{\nabla}_{e_1}e_1\in\Pi^\perp,
\end{align*}
where we used lemma \ref{lem:commuting} in the second line.
It follows that $\Pi^\perp$ is locally constant, hence constant on $\Sigma$.
\end{proof}

We now parametrize $\Sigma$ by two parameters $s$ and $u$ so that $(0,0)$ corresponds to the point $x_0\in \Sigma$, $\frac{\partial x}{\partial u}=E_2=e_2$, and $\frac{\partial x}{\partial s}=E_1=\frac{e_1}{\sqrt{\mu}}$ (this is possible since $E_1$ and $E_2$ commute by lemma \ref{lem:commuting}).

Write
\begin{align*}
   r=\frac{g}{\sqrt{C}}=\frac{\mu^{-\frac 12}}{\sqrt{C}},
\end{align*}
where $C$ is the constant in \eqref{eq:3.17}. Note that $g,r$ only depends on $u$, since $e_1(\mu)=0$. We have $ \displaystyle{\frac{r'}{r}=-\frac{\mu'}{2\mu}}$ (here $r'=\dfrac{dr}{du}$) and
\begin{align}\label{eq:3.22}
\frac{r''}{r}+1+\lambda_1\lambda_2=0,\quad (r')^2+r^2(1+\lambda_1^2)=1.
\end{align}

It follows from lemma \ref{lem:perp-plane} that the two-dimensional plane $\Pi$ perpendicular to $\Pi^\perp$ is also constant.  This is generated by the orthonormal basis $\{p,q\}$ where $p=\frac{\lambda_1x-\nu}{\sqrt{1+\lambda_1^2}}$ and $q=\displaystyle{\frac{r(1+\lambda_1^2)e_2-r'x-\lambda r'\nu}{\sqrt{1+\lambda^2}}}$.   We then have a convenient orthonormal basis for $\RR^4$ given as follows:
$$
v_1=e_1;\quad v_2=\frac{\bar{\nabla}_{e_1}e_1}{|\bar{\nabla}_{e_1}e_1|} = -r'e_2-\lambda r\nu-rx
$$
forming an orthonormal basis for $\Pi^\perp$; and
$$
v_3=p; \quad v_4=q
$$
forming an orthonormal basis for $\Pi$.  We compute the rates of change of these bases along the vector fields $E_1$ and $E_2$:  We have
\begin{equation}\label{eq:d1v1}
E_1v_1 = \frac{1}{\sqrt{\mu}}
\bar{\nabla}_{e_1}e_1 = \frac{1}{r\sqrt{\mu}}v_2=\sqrt{C}v_2;
\end{equation}
and (since the plane $\Pi^\perp$ is preserved and $v_1$ and $v_2$ are orthogonal)
\begin{equation}\label{eq:d1v2}
E_1v_2 = -\sqrt{C}v_1.
\end{equation}

%
We also have
\begin{align*}
E_1v_3 &= E_1p\\
&= E_1\left(\frac{\lambda_1 x-\nu}{\sqrt{1+\lambda_1^2}}\right)\\
&= \frac{1}{\sqrt{\mu(1+\lambda_1^2)}}e_1(\lambda_1 x-\nu)\\
&= \frac{1}{\sqrt{\mu(1+\lambda_1^2)}}(\lambda_1e_1-\lambda_1e_1)\\
&=0
\end{align*}
(since $e_1\lambda_1=0$).   Since $\Pi$ is constant it follows that $E_1v_4=E_1q=0$ also.
Now we consider the $E_2$ direction:  We have
\begin{equation}\label{eq:d2v1}
E_2v_1 = \bar{\nabla}_{e_2}e_1 = 0,
\end{equation}
and hence also $E_2v_2=0$.   Finally, a direct calculation shows that we have $E_2v_3=\alpha v_4$ and $E_2v_4=-\alpha v_3$, where $\alpha = \frac{2\mu}{r(1+\lambda_1^2)}$.  We observe that $\alpha$ is positive everywhere.

\begin{lemma}\label{circle-flow}
Let ${\mathbf a}=v_1(0,0)$, ${\mathbf b}=v_2(0,0)$, ${\mathbf c}=v_3(0,0)$ and ${\mathbf d}=v_4(0,0)$, and define $A(u)=\int_0^u\alpha(\tau)\,d\tau$.  Then
\begin{align*}
v_1(s,u) &= {\mathbf a}\cos(\sqrt{C} s) + {\mathbf b}\sin(\sqrt{C} s);\\
v_2(s,u) &= -{\mathbf a}\sin(\sqrt{C} s)+{\mathbf b}\cos(\sqrt{C} s);\\
v_3(s,u) &= {\mathbf c}\cos(A(u))+{\mathbf d}\sin(A(u));\\
v_4(s,u) &= -{\mathbf c}\sin(A(u))+{\mathbf d}\cos(A(u)).
\end{align*}
\end{lemma}

\begin{proof}
We have $\frac{\partial}{\partial u}v_i=0$ for $i=1,2$, and by
\eqref{eq:d1v1} and \eqref{eq:d1v2}
$$
\frac{\partial}{\partial s}v_1=\sqrt{C} v_2;\quad \frac{\partial}{\partial s}v_2=-\sqrt{C} v_1.
$$
The argument for $v_3$ and $v_4$ is similar.
\end{proof}

Now we can complete the proof of Theorem \ref{thm:symmetry}:  From the definitions of $v_1,\dots,v_4$ we have
$$
x(s,u)=-rv_2+\frac{\lambda}{\sqrt{1+\lambda^2}}v_3-\frac{r'}{\sqrt{1+\lambda^2}}v_4.
$$
Define $B(u)$ by
\begin{align*}
    \cos B(u)=\frac{\lambda}{\sqrt{1+\lambda^2}\sqrt{1-r^2}},\quad \sin B(u)=\frac{r'}{\sqrt{1+\lambda^2}\sqrt{1-r^2}}.
\end{align*}
From lemma \ref{circle-flow}, we have
\begin{align*}
    x(s,u)=&r(u)\sin(\sqrt{C}s){\mathbf a}-r(u)\cos(\sqrt{C}s){\mathbf b}\\
    &\qquad +\sqrt{1-r^2(u)}\cos(A(u)-B(u)){\mathbf c}+\sqrt{1-r^2(u)}\sin(A(u)-B(u)){\mathbf d}.
\end{align*}
Now choosing
\begin{align*}
    {\mathbf a}=&v_1(0,0)=(0,1,0,0)\\
    {\mathbf b}=&v_2(0,0)=(-1,0,0,0)\\
    {\mathbf c}=&v_3(0,0)=(0,0,\cos B(0),\sin B(0))\\
    {\mathbf d}=&v_4(0,0)=(0,0,-\sin B(0),\cos B(0)),
\end{align*}
then we  we get
\begin{align*}
    x(s,u)=&\biggl(r(u)\cos(\sqrt{C}s), r(u)\sin(\sqrt{C}s),\sqrt{1-r^2(u)}\cos\theta(u),\sqrt{1-r^2(u)}\sin\theta(u)\biggr),
\end{align*}
where $\theta(u)=A(u)-B(u)+B(0)$. Since $\theta(0)=0$ and (note that $B(u)=\arctan\dfrac{r'}{\lambda}$)
\begin{align*}
    \theta'(u)=&A'(u)-B'(u)\\
    =&\frac{2\mu}{(1+\lambda^2)r}-\frac 1{1+(\frac{r'}{\lambda})^2}\biggl(\frac{r''}{\lambda}-\frac{r'\lambda'}{\lambda^2}\biggr)\\
    =&\frac{\lambda(u) r(u)}{1-r^2(u)},
\end{align*}
that is $\displaystyle{\theta(u)=\int_0^u\frac{r(\tau)\lambda(\tau)}{1-r^2(\tau)}d\tau}$ (see \cite{Per}). Let $v=\sqrt{C}s$, we conclude that $\Sigma$ is given by (also compare with \cite{Ot1,Per})
\begin{align}
    F(u,v)=\biggl(r(u)\cos v,r(u)\sin v,\sqrt{1-r^2(u)}\cos\theta(u),\sqrt{1-r^2(u)}\sin\theta(u)\biggr),
\end{align}
where $0\leq v<2\pi$, $0\leq u< \frac{m\pi}{\sqrt{1+H^2}}$ and $m$ is some positive integer.

\end{proof}

\section{Monotonicity of the period}

In this section we complete the proof of Theorem \ref{main.theorem} by proving that the period function $K(H,C)$ is monotone in $C$, which implies that there can be at most one embedded CMC torus (up to congruence) with given values of $H$ and $m$.

Define the number (see \cite{Per})
\begin{equation}\label{eq:4.1}
K(H,C)=\int_{\frac{t^2_1}{C}}^{\frac{t^2_2}{C}} \frac{(Hu+C^{-1})}
{\sqrt{u}(1-u)\sqrt{-u^2(1+H^2)+(1-2HC^{-1})u-C^{-2})}}du,
\end{equation}
where $C$ is a constant greater than $2(H+\sqrt{1+H^2})$  and $t_1$ and $t_2$ are defined by \eqref{eq:3.18}.  We need the following result (see Theorem 3.1 of \cite{Per}, also \cite{DD}, \cite{BL}, \cite{LW}, or \cite{WCL})

\begin{proposition}
Suppose that $F: \Sigma \to S^3$ is a rotational torus in $S^3$, which is not a Clifford torus and is given by \eqref{eq:3.19} Then $F(\Sigma)$ is an embedded torus if and only if
\begin{equation}\label{eq:4.2}
K(H,C)=\frac{2\pi}{m}
\end{equation}
for some positive integer $m$.
\end{proposition}

We can check directly
\begin{equation}\label{eq:4.3}
K(H,C)\to
\frac{\sqrt{2}\pi}{(1+H^2)^{\frac{1}{4}}(H+\sqrt{1+H^2})^{\frac{1}{2}}},\;\; {\rm when}\;\;
C\to a(H)
\end{equation}
where
\begin{equation}\label{eq:4.4}
a(H)=2(H+\sqrt{1+H^2})^+.
\end{equation}
and
\begin{equation}\label{eq:4.5}
K(H,C)\to 2\text{\rm arccot}(H),\;\;\; {\rm when}\;\;
C\to \infty
\end{equation}
The following result can be found in  Perdomo's paper in \cite{Per}) (also see Ripoll's paper \cite{Ri})

\begin{proposition} If $H\not=0, \pm \frac{1}{\sqrt{3}}$, there exist compact embedded tori in $S^3$ with constant mean curvature $H$, which are not Clifford tori.
In fact, for any integer $m\geq 2$, if $H$ satisfies
\begin{equation}\label{eq:4.6}
cot\frac{\pi}{m}<H<\frac{m^2-2}{2\sqrt{m^2-1}},
\end{equation}
then there exists a compact embedded torus in $S^3$ with constant mean curvature $H$ whose isometry group contains $O(2)\times Z_m$ which is not a Clifford torus.
\end{proposition}

Now we prove the following result

\begin{proposition}\label{prop:monotonic}
 For any nonnegative real number $H$, $K(H,C)$ is monotone decreasing in $2(H+\sqrt{1+H^2})<C<\infty$.
\end{proposition}

\begin{remark} When $H=0$, Proposition \ref{prop:monotonic} was proved by T. Otsuki in \cite{Ot2}. We note that our proof here is simpler than Ostuki's \cite{Ot2} even in minimal case.
\end{remark}

\begin{proof}

Denote $x_1=\frac{t_1^2}{C},x_2=\frac{t_2^2}{C}$.  By \eqref{eq:3.18}, we note that $x_1$ and $x_2$ satisfy
\begin{equation}\label{eq:4.7}
0< x_1\equiv \frac{t_1^2}{C} < x_2\equiv \frac{t_2^2}{C}<1.
\end{equation}

Let $a=C^{-1}$, and write $T(H,a)=K(H,\frac{1}{a})$.  To prove $K'(H,C)<0$, it suffices to show $T'(H,a)>0$, where the prime denotes the derivative with respect to $a$. $T(H,a)$ is given by
$$
T(H,a)=\int_{x_1}^{x_2}\frac{Hu+a}{(1-u)\sqrt{u}\sqrt{1+H^2}\sqrt{(u-x_1)(x_2-u)}}du.
$$

On the region
$$\Omega=\mathbb{C}-\{z\in \mathbb{C}| z\leq 0, \ or\ x_1\leq z\leq x_2, \ or \ z=1\},$$
the function
$$
\frac{Hz+a}{(1-z)f(z)}\label{eq:4.7}
$$
is well defined and holomorphic, where
$$
f(z)=-\sqrt{1+H^2}i\sqrt{-z}(x_2-z)\sqrt{\frac{z-x_1}{x_2-z}}.
$$

We note that
$$
f(z)^2=(1+H^2)z(x_2-z)(z-x_1)=z[-(1+H^2)z^2+(1-2Ha)z-a^2].
$$

\begin{figure}
\includegraphics[width=5in]{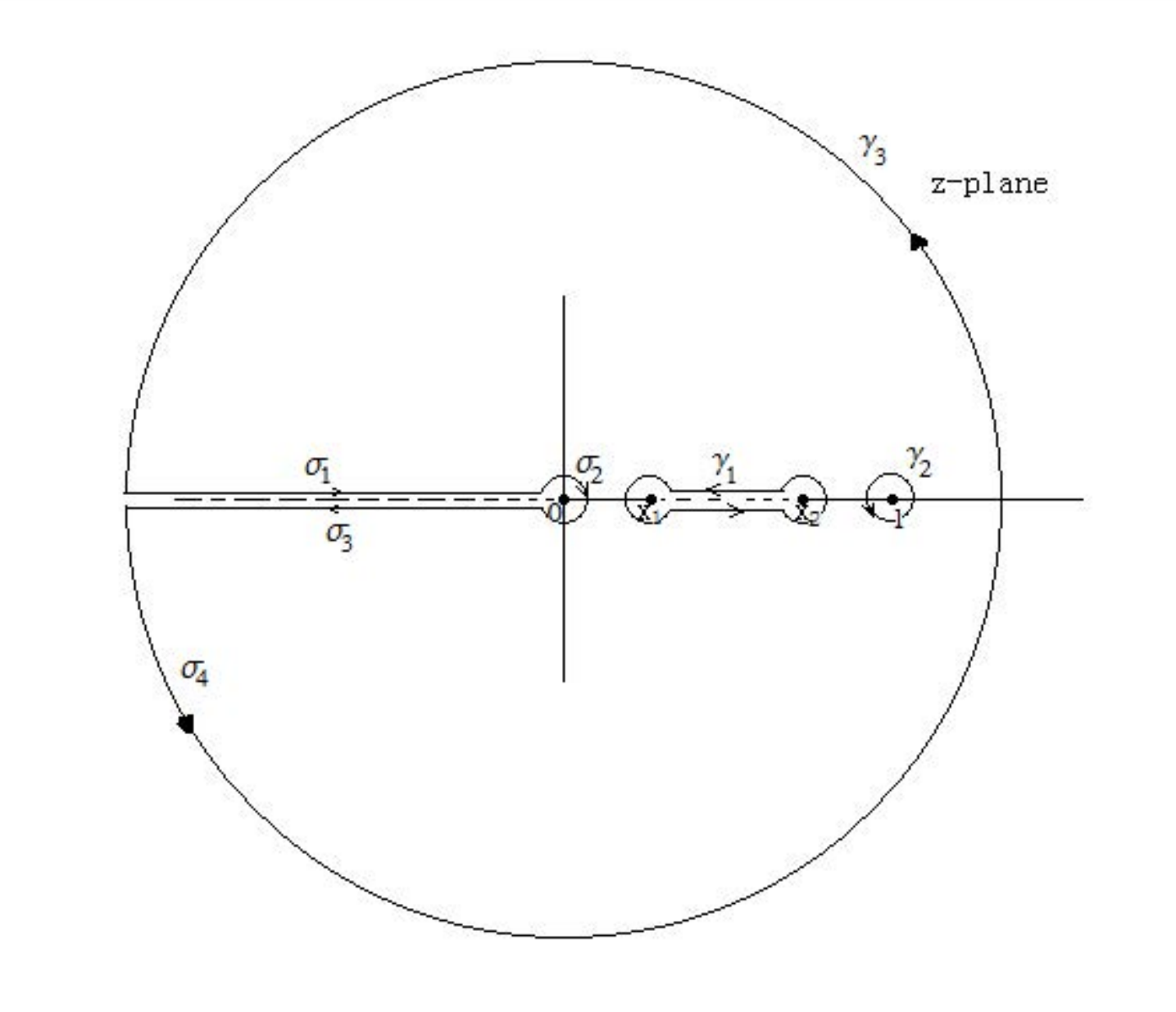}
\caption{Paths of integration for the proof of proposition \ref{prop:monotonic}}
\end{figure}

Now, choose a closed curve $\gamma_1$ in $\Omega$ such that its interior contains the only two critical points $x_1=\frac{t_1^2}{C},x_2=\frac{t_2^2}{C}$. We can assume that $\gamma_1$ is chosen as in the figure.

In analogy with Otsuki's observation (\cite{Ot2}), we can check
\begin{equation}\label{eq:4.8}
T(H,a)=-\frac{1}{2} \int_{\gamma_1} \frac{Hz+a}{(1-z)f(z)} dz.
\end{equation}

From \eqref{eq:4.8}, we can check by a direct calculation
\begin{equation}\label{eq:4.9}
T'(H,a)=\frac{dT(H,a)}{da}=
-\frac{1}{2} \int_{\gamma_1}\frac{z^2}{f(z)^3}dz.
\end{equation}

Next we deform the curve $\gamma_1$ to another curve $\gamma_3$ which will pass through a critical point $z=1$, which will cause a difference of $\pi$, a constant. Let $\gamma_3$ be the closed curve shown in the picture, a large circle with radius $R$ large (see $\sigma_4$), a small circle with radius $r$ and with center at $0$ (see $\sigma_2$), two straight line close to the half line $x<0$ on the real axis (see $\sigma_1$ and $\sigma_3$).
Then we have
\begin{equation}\label{eq:4.10}
T(H,a)=-\pi-\frac{1}{2} \int_{\gamma_3} \frac{Hz+a}{(1-z)f(z)} dz.
\end{equation}

Thus we have from \eqref{eq:4.10} and \eqref{eq:4.9}
\begin{equation}\label{eq:4.11}
T'(H,a)=-\frac{1}{2} \int_{\gamma_3}\frac{z^2}{f(z)^3}dz.
\end{equation}

The following identities can be checked directly

\begin{equation}\label{eq:4.12}
\begin{aligned}
&\int_{\sigma_1}
\frac{z^2}{f(z)^3}dz\\
&=\int_{\sigma_3}
\frac{z^2}{f(z)^3}dz\\
&=\int_{-\infty}^{0} -\frac{x^2}{\left(\sqrt{x(-x^2(1+H^2)+(1-2Ha)x-a^2)}\right)^3}dx\\
&=
\int_{-\infty}^{0}-\frac{x^2}{\left(\sqrt{(1+H^2)(-x)(x_1-x)(x_2-x)}\right)^3}dx
\end{aligned}
\end{equation}
Here the only thing we need to explain is how the minus sign arises in the last line of \eqref{eq:4.12}. On $\sigma_1$, we have $x<0, y\to 0^+$, therefore
$$
\sqrt{-z}=\sqrt{-x-iy}\to -\sqrt{-x},\;\;
\sqrt{\frac{z-x_1}{x_2-z}}\to i\sqrt{\frac{x_1-x}{x_2-x}},
$$
$$
f(z)\to -\sqrt{1+H^2}\sqrt{-x}\sqrt{(x_1-x)(x_2-x)}.
$$

We also have
\begin{equation}\label{eq:4.13}
\lim\limits_{r\rightarrow 0}\int_{\sigma_2}\frac{z^2}{f(z)^3}dz=0
\end{equation}
and
\begin{equation}\label{eq:4.14}
\lim\limits_{R\rightarrow \infty}\int_{\sigma_4}
\frac{z^2}{f(z)^3}dz=0
\end{equation}
Combining \eqref{eq:4.12}, \eqref{eq:4.13} and \eqref{eq:4.14}, we obtain
\begin{equation}\label{eq:4.15}
\begin{aligned}
&T'(H,a)=-\frac{1}{2}\int_{\gamma_3}
\frac{z^2}{f(z)^3}dz\\
&=-\frac{1}{2}\left(\int_{\sigma_1}+\int_{\sigma_2}+\int_{\sigma_3}+\int_{\sigma_4}\right)
\frac{z^2}{f(z)^3}dz\\
&=
\int_{-\infty}^{0}\frac{x^2}{\left(\sqrt{(1+H^2)(-x)(x_1-x)(x_2-x)}\right)^3}dx
>0.
\end{aligned}
\end{equation}
\end{proof}

\begin{proof}[Proof of Theorem \ref{main.theorem}] By Theorem \ref{thm:symmetry}, every embedded CMC torus is a surface of rotation, and by the results of \cite{Per} this is produced from a solution of equation \eqref{eq:3.16} for some value of $C$ such that \eqref{eq:4.2} holds for some $m\geq 2$.  The monotonicity of $K(H,C)$ implies that this occurs for at most one value of $C$, and the limits \eqref{eq:4.3} and \eqref{eq:4.5} imply that there exists such a $C$ if and only if
\begin{equation}\label{eq:4.16}
2\text{\rm arccot}(H)< \frac{2\pi}{m}<\frac{\sqrt{2}\pi}{(1+H^2)^{\frac{1}{4}}(H+\sqrt{1+H^2})^{\frac{1}{2}}}.
\end{equation}

If $H=0$, we have that $K(0,C)$ takes values for $2<C<\infty$ in the range
\begin{equation}\label{eq:4.17}
\pi<K(0,C)<\sqrt{2}\pi,
\end{equation}
so there exists no integer $m\geq 2$ satisfying
$K(0,C)=\frac{2\pi}{m}$, and there are no compact embedded minimal tori in $S^3$ other than the Clifford tori. This was proved for the rotational symmetric case by Otsuki \cite{Ot1} and in general by Brendle \cite{Br}.

If $H=\pm \frac{1}{\sqrt{3}}$, we can assume $H=\frac{1}{\sqrt{3}}$ by reversing the unit normal vector if necessary. In this case $K(\frac{1}{\sqrt{3}},C)$ takes values for $2\sqrt{3}<C<\infty$ in the range
\begin{equation}\label{eq:4.18}
\frac{2}{3}\pi<K(\frac{1}{\sqrt{3}},C)<\pi,
\end{equation}
thus there exists no integer $m\geq 2$ such that
$K(\frac{1}{\sqrt{3}},C)=\frac{2\pi}{m}$, and consequently there are no compact embedded torus in $S^3$ with $H=\frac{1}{\sqrt{3}}$, other than the Clifford torus.  This rigidity was previously unknown even in the rotationally symmetric case, though it is suggested by the results of Perdomo \cite{Per} and Ripoll \cite{Ri}.

For all other values of $H$ there exists some $m$ such that equation \eqref{eq:4.2} holds for some $C$, and consequently there always exist embedded CMC tori which are not congruent to Clifford tori.  The number of these (up to congruence) is precisely the number of values of $m$ for which \eqref{eq:4.16} holds.

This completes the proof of theorem \ref{main.theorem}.
\end{proof}

\end{document}